\documentclass[a4paper,11pt]{amsart}
\usepackage{hyperref,fancyhdr,mathrsfs,amsmath,amscd,amsthm,amsfonts,latexsym,amssymb,stmaryrd}
\usepackage[all]{xy}

\DeclareMathOperator{\rank}{rank}

\DeclareMathOperator{\dif}{d}

\newcommand{\Cal}{\mathcal{C}}

\newcommand{\ol}{\mathcal{O}}

\def \a{\alpha}

\def \phi{\varphi}
\def \Phi{\varPhi}
\def \p{\pi}
\def \r{\rho}

\def \C{\mathbb{C}\,}

\def\widecheckg{g^{\hspace*{-2.5pt}\vbox to 5pt{\hbox to
0pt{\LARGE$\check{}$}}}\hspace*{2pt}}

\def\widecheckl{\lambda^{\hspace*{-3.5pt}\vbox to 8pt{\hbox to
0pt{\LARGE$\check{}$}}}\hspace*{2pt}}

\begin{document}

\title{Twistorial structures revisited}
\author{Radu Pantilie}  
\email{\href{mailto:radu.pantilie@imar.ro}{radu.pantilie@imar.ro}}
\address{R.~Pantilie, Institutul de Matematic\u a ``Simion~Stoilow'' al Academiei Rom\^ane,
C.P. 1-764, 014700, Bucure\c sti, Rom\^ania}

\newtheorem{thm}{Theorem}[section]
\newtheorem{lem}[thm]{Lemma}
\newtheorem{cor}[thm]{Corollary}
\newtheorem{prop}[thm]{Proposition}

\theoremstyle{definition}

\newtheorem{defn}[thm]{Definition}
\newtheorem{rem}[thm]{Remark}
\newtheorem{exm}[thm]{Example}

\numberwithin{equation}{section}

\begin{abstract}
We review the twistorial structures by providing a setting under which the corresponding (differential) geometry can be described, 
by involving the $\r$-con\-nec\-tions. This applies, for example, to give new proofs of the existence of the relevant connections for  
the projective and the quaternionic geometries. Along the way, we show that, in this setting, the Ward transformation is a consequence 
of the good behaviour of the $\r$-connections, under pull back.  
\end{abstract}

\maketitle
\thispagestyle{empty}
\vspace{-3mm}

\section*{Introduction} 

\indent 
A twistor space is a complex manifold $Z$ endowed with a family $\{Y_x\}_{x\in M}$ of compact complex submanifolds 
such that the canonical map $\psi:Y=\bigsqcup_{x\in M}Y_x\to Z$ is a surjective submersion. 
The main task is to describe the (differential) geometry induced on $M$, and to see, to what extent, from this geometry we 
can retrieve $Z$\,.\\ 
\indent 
The first examples are provided by the projective and the conformal structures (see \cite{LeB-thesis}\,, \cite{LeB-nullgeod}\,), 
where $Y$ is $PTM$ and a bundle of hyperquadrics, respectively. On the other hand, we have the quaternionic geometry, 
where $Y$ is a Riemann sphere bundle (see \cite{Pan-qgfs}\,).\\ 
\indent 
To tackle the main task, we, firstly, turn to the canonical projection $\p:Y\to M$ and observe that, for any $x\in M$, the differential of $\p$ 
determines a map from $Y_x$ to the Grassmannian of $T_xM$\,. This is how the `linear twistorial structures', which we consider 
in Section \ref{section:lin_twist_str}\,, appear.\\ 
\indent 
Next, we have to understand that the classical connections are not sufficient to model even the simple examples given by the projective spaces, endowed 
with the corresponding families of Veronese curves. We are, thus, led to consider the $\r$-connections \cite{Pan-qgfs}\,. We have, however, to make 
some further assumptions, such as `regularity' which means that the spaces of sections of the duals of $({\rm ker}\dif\!\psi)|_{Y_x}$\,, $(x\in M)$\,,  
have the same dimension. This supplies, through direct image, an enriched tangent bundle, that is, a vector bundle $E$ over $M$ 
together with a morphism of vector bundles $\r:E\to TM$. The relevant notions and results are given in Section \ref{section:regular_a_twist}\,.\\ 
\indent 
Finally, in Section \ref{section:regular_twist}\,, we study a class of twistor spaces which provides a generalization of the $\r$-quaternionic structures 
of \cite{Pan-qgfs} (Theorem \ref{thm:suff_cond} and Remark \ref{rem:gen_rho_q}\,). The idea, basic in twistor theory, is that any (fairly good) line bundle 
over $Z$ induces, through pull back by $\psi$ and direct image through $\p$, a vector bundle over $M$ endowed with a $\r$-connection.\\ 
\indent   
Along the way, we, also, show that, for regular twistorial structures, the Ward transformation is a consequence of the 
good behaviour of the $\r$-connections, under pull back (Corollary \ref{cor:Ward}\,).

\section{Linear twistorial structures} \label{section:lin_twist_str} 

\indent 
We work in the category of complex manifolds. 

\begin{defn} 
Let $U$ be a (complex) vector space, let $Y$ be a compact manifold, and let $k\in\mathbb{N}$\,, $k\leq\dim U$.
A \emph{linear twistorial $Y\!$-structure} on $U$ is a map $\phi:Y\to{\rm Gr}_kU$.  
Let 
\begin{equation} \label{e:lin_Y-str} 
0\longrightarrow\mathcal{U}_-\longrightarrow Y\times U\longrightarrow\mathcal{U}_+\longrightarrow0 
\end{equation} 
be the pull back through $\phi$ of the tautological exact sequence of vector bundles over ${\rm Gr}_kU$. 
We say that \eqref{e:lin_Y-str} (which, obviously, determines $\phi$) is \emph{the exact sequence (of vector bundles) of $\phi$}\,.     
If the associated cohomology exact sequence gives a linear isomorphism between $U$ and the space of sections 
of $\mathcal{U}_+$ we say that $\phi$ is \emph{maximal} (cf.\ \cite{Kod}\,). 
\end{defn} 

\indent 
For brevity, we shall use the expression `linear $Y$-structure' instead of `linear twistorial $Y$-structure'.\\ 
\indent 
The vector spaces endowed with linear $Y\!$-structures form a category in an obvious way (we can even let $Y$ to change from 
one vector space to another, if necessary).\\ 
\indent 
Let $\phi:Y\to{\rm Gr}_kU$ be a maximal linear $Y\!$-structure on $U$. Note that, $\phi$ must be nonconstant if $Y$ is not a point. 
Furthermore, $H^0\bigl(\mathcal{U}_-\bigr)=0$ and, under the further assumption that $H^1(\ol_Y)=0$\,, then, also 
$H^1\bigl(\mathcal{U}_-\bigr)=0$\,. Conversely, the exact sequence \eqref{e:lin_Y-str} corresponds 
to a maximal linear $Y\!$-structure if $H^0\bigl(\mathcal{U}_-\bigr)=0$ and $H^1\bigl(\mathcal{U}_-\bigr)=0$\,.\\ 
\indent 
Now, dualizing the exact sequence \eqref{e:lin_Y-str} of a linear $Y\!$-structure $\phi$ on $U$  
we obtain the exact sequence of a linear $Y\!$-structure on $U^*$. However, the dual of a maximal linear $Y\!$-structure is not necessarily maximal. 
For this reason, let $E$ be the dual of the space of 
sections of $\mathcal{U}_-^*$ and let $\r$ be the transpose of the linear map from $U^*$ to $E^*$ given by the cohomology exact sequence 
of the dual of \eqref{e:lin_Y-str}\,. As $\mathcal{U}_-^*$ is generated by its (global) sections, the obvious morphism of vector bundles 
$Y\times E^*\to\mathcal{U}_-^*$ is surjective and therefore we obtain an exact sequence 
\begin{equation} \label{e:lin_rho_induced} 
0\longrightarrow\mathcal{U}_-\longrightarrow Y\times E\longrightarrow\mathcal{E}\longrightarrow0\;,  
\end{equation} 
for some vector bundle $\mathcal{E}$ over $Y$. Moreover, $\r$ induces a morphism from \eqref{e:lin_rho_induced} to \eqref{e:lin_Y-str}\,; 
in particular, a morphism $\mathcal{R}:\mathcal{E}\to\mathcal{U}_+$ from which, if $H^0\bigl(\mathcal{U}_-\bigr)=0$\,, 
we retrieve $\r$ when passing to the spaces of sections.  

\begin{rem} 
1) The dual of \eqref{e:lin_rho_induced} corresponds to a maximal linear $Y$-structure.\\ 
\indent 
2) From the dual of \eqref{e:lin_Y-str} we deduce that the cokernel of $\r$ is the dual of $H^0\bigl(\mathcal{U}_+^*\bigr)$\,. 
Furthermore, if $H^1(\ol_Y)=0$ then the kernel of $\r$ is the dual of $H^1\bigl(\mathcal{U}_+^*\bigr)$\,. 
\end{rem} 

\begin{exm} 
A vector bundle over $\C\!P^1$ appears as the corresponding 
$\mathcal{U}_+$ of a linear $\C\!P^1$-structure if and only if it is nonnegative.\\ 
\indent  
If the corresponding map is an embedding, we retrieve the linear quaternionic-like structures of \cite{Pan-hqo}\,,  
whilst maximal linear $\C\!P^1$-structure are just the $\r$-quaternionic structures of \cite{Pan-qgfs}\,. 
Then the dual of the corresponding \eqref{e:lin_Y-str}\,, also, corresponds 
to a maximal linear $\C\!P^1$-structure if and only if in the Birkhoff--Grothendieck decomposition of $\mathcal{U}_+$ 
appear only terms of Chern number one; that is, if and only if the given linear $\C\!P^1$-structure gives a linear (classical, complex) 
quaternionic structure.  
\end{exm} 

\begin{exm} \label{exm:lin_Y-str} 
Let $Y$ be endowed with an ample line bundle $L$\,. Then from the Kodaira vanishing theorem and Serre duality we obtain 
$H^0(L^*)=0$\,; furthermore, if $\dim Y\geq2$ then, also, $H^1(L^*)=0$\,.\\ 
\indent 
Therefore if $\dim Y\geq2$\,, $H^1(\ol_Y)=0$ and $L$ has empty base locus, the map from $Y$ into the projectivisation of the dual of the space 
of sections of $L$ is a maximal linear $Y\!$-structure. 
Moreover, the dual of the corresponding \eqref{e:lin_Y-str}\,, also, corresponds to a maximal linear $Y\!$-structure. 
Furthermore, on tensorising $L$ with a trivial vector bundle then we again obtain the exact sequence 
of a linear $Y\!$-structure whose dual, also, corresponds to a linear $Y\!$-structure. 
\end{exm}

\section{Regular almost twistorial structures} \label{section:regular_a_twist} 

\indent 
Twistor theory imposes the use of the category whose objects are triples $(M,E,\r)$\,, 
where $M$ is a manifold, $E$ is a vector bundle over $M$, and $\r:E\to TM$ is a morphism of vector bundles. 
The morphisms between two such objects $(M,E,\r)$ and $(M',E',\r')$ are pairs $(\phi,\Phi)$\,, where 
$\phi:M\to M'$ is a map, and $\Phi:E\to E'$ is a vector bundles morphism, over $\phi$\,, such that $\r'\circ\Phi=\dif\!\phi\circ\r$\,.\\ 
\indent 
The notion of connection adapts, accordingly, to this setting. Let $(P,M,G)$ be a principal bundle, and let $\r:E\to TM$ be a 
morphism of vector bundles, where $E$ is a vector bundle over $M$. A \emph{principal $\r$-connection} \cite{Pan-qgfs} on $P$ 
is a morphism of vector bundles $c:E\to TP/G$ such that when composed with the canonical morphism of vector bundles 
$TP/G\to TM$ gives $\r$\,.\\ 
\indent 
We, also, have the notion of associated $\r$-connections which for vector bundles corresponds to covariant $\r$-derivations.  

\begin{prop} \label{prop:rho-pull-back}
Let $(\phi,\Phi):(M,E,\r)\to(M',E',\r')$ be a morphism and let $(P,M',G)$ be a principal bundle.\\ 
\indent  
For any principal $\r'$-connection $c'$ on $P$ there exists a unique principal $\r$-connection $c$ on $\phi^*P$ 
such that $c'\circ\Phi=\widetilde{\phi}\circ c$\,, where $\widetilde{\phi}:T(\phi^*P)/G\to TP/G$ is the morphism of vector bundles 
induced by (the differential of) $\phi$\,. 
\end{prop} 
\begin{proof} 
Let $\p:P\to M'$ be the projection. As $\phi^*P$ embedded into $M\times P$ is formed of those pairs $(x,u)$ with $\phi(x)=\p(u)$\,, 
we have that $T(\phi^*P)/G$ embeds into $TM\times(TP/G)$ such that be formed of those pairs $(X,U)$ satisfying 
$\dif\!\phi(X)=\widetilde{\dif\!\pi}(U)$\,, where $\widetilde{\dif\!\pi}:TP/G\to TM'$ is the morphism of vector bundles induced by $\dif\!\p$\,.\\ 
\indent  
On the other hand, $\widetilde{\dif\!\pi}\circ c'\circ\Phi=\r'\circ\Phi=\dif\!\phi\circ\r$\,. Consequently, $\r$ and $c'\circ\Phi$ uniquely determine 
a vector bundles morphism $c:E\to T(\phi^*P)/G$, as required. 
\end{proof} 

\indent 
The principal $\r$-connection obtained in Proposition \ref{prop:rho-pull-back} is called \emph{the pull back by $(\phi,\Phi)$ of $c'$}. 

\begin{rem} 
The association $c'\mapsto c$ of Proposition \ref{prop:rho-pull-back} is a morphism of (possibly empty) affine spaces. 
The involved vector spaces are the spaces of sections of ${\rm Hom}(E',{\rm Ad}P)$ and ${\rm Hom}\bigl(E,{\rm Ad}(\phi^*P)\bigr)$, respectively. 
The corresponding linear map is induced by $\Phi$ through the identification ${\rm Ad}(\phi^*P)=\phi^*({\rm Ad}P)$\,. 
\end{rem} 

\indent 
Now, recall \cite{PanWoo-sd} that, an \emph{almost twistorial structure} on a manifold $M$ is a triple $(Y,\p,\Cal)$\,, where 
$\p:Y\to M$ is a proper surjective submersion, and $\Cal$ is a distribution on $Y$ such that $\Cal\cap({\rm ker}\dif\!\p)=0$\,.\\ 
\indent 
Let $(Y,\p,\Cal)$ be an almost twistorial structure on $M$. For each $x\in M$, denote $Y_x=\p^{-1}(x)$\,, and note that we have a map 
$Y_x\to{\rm Gr}_k(T_xM)$\,, $y\mapsto\dif\!\p\bigl(\Cal_y\bigr)$\,, $\bigl(y\in Y_x\bigr)$\,, where $k=\rank\Cal$.  
If this map is a maximal linear $Y_x$-structure on $T_xM$, for any $x\in M$,  
we say that $(Y,\p,\Cal)$ is \emph{maximal}.\\ 
\indent  
Let  
\begin{equation} \label{e:Y_str} 
0\longrightarrow\Cal\longrightarrow \p^*(TM)\longrightarrow\mathcal{T}\longrightarrow0 
\end{equation} 
be the exact sequence of vector bundles over $Y$ induced, through pull back, by the tautological exact sequence over 
${\rm Gr}_k(TM)$\,. Obviously, we have an isomorphism between $\mathcal{T}$ and the quotient of $TY$ through $\Cal+({\rm ker}\dif\!\p)$\,.\\  
\indent 
We are interested in the direct image through $\p$ of the dual of $\Cal$\,. This is a vector bundle over $M$ if 
$x\mapsto h^0\bigl(\Cal^*_x\bigl)$ is constant (see \cite[p.\ 211]{GraRem-cas}\,), where we have denoted $\Cal_x=\Cal|_{Y_x}$\,;  
then we say that $(Y,\p,\Cal)$ is \emph{regular}. Assuming this, on denoting by $E$ the dual of the space of sections of the induced vector bundle, 
and by dualizing \eqref{e:Y_str} we, also, obtain a morphism of vector bundles $\r:E\to TM$. 
Furthermore, in this context, \eqref{e:lin_rho_induced} becomes 
\begin{equation} \label{e:rho_induced} 
0\longrightarrow\Cal\longrightarrow\p^*E\longrightarrow\mathcal{E}\longrightarrow0\;, 
\end{equation} 
for some vector bundle $\mathcal{E}$\,. Denote by $\chi$ the vector bundles morphism, over $\p$\,, obtained as the composition 
of $\Cal\to\p^*E$ followed by the canonical bundle map $\p^*E\to E$\,.\\  
\indent 
Thus, on denoting by $\iota:\Cal\to TY$ the inclusion, we have obtained a morphism $(\p,\chi):(Y,\Cal,\iota)\to(M,E,\r)$\,.\\ 
\indent 
The following result is fundamental in twistor theory. 

\begin{prop} \label{prop:almost_Ward} 
Let $(Y,\p,\Cal)$ be a regular almost twistorial structure on $M$, let $E$ be the direct image of $\Cal^*$ by $\p$, 
denote by $\r$ the induced morphism of bundles from $E$ to $TM$, and let $(\p,\chi):(Y,\Cal,\iota)\to(M,E,\r)$ 
be the resulting morphism, where $\iota:\Cal\to TY$ is the inclusion.\\ 
\indent 
For any principal bundle $P$ over $M$, the pull back by $(\p,\chi)$ establishes an isomorphism between the following affine spaces:\\ 
\indent 
\quad{\rm (i)} the space of principal $\r$-connections on $P$,\\ 
\indent 
\quad{\rm (ii)} the space of principal partial connections on $\p^*P$, over $\Cal$\,. 
\end{prop}   
\begin{proof} 
It is sufficient to consider the case $P=M\times G$. Then any principal $\r$-connection on $P$ is given by its (local) connection form 
which is a section of $\mathfrak{g}\otimes E^*$, where $\mathfrak{g}$ is the Lie algebra of $G$. 
On the other hand, as $\p^*P=Y\times G$, any principal partial connection on $\p^*P$, over $\Cal$\,, is given by a section of 
$\mathfrak{g}\otimes\Cal^*$.\\ 
\indent 
Accordingly, the pull back by $(\p,\chi)$ associates to any section $A$ of $\mathfrak{g}\otimes E^*$ the restriction to $\Cal$ of $\p^*A$\,.\\ 
\indent  
Now, as $(Y,\p,\Cal)$ is regular, and $E$ is the direct image through $\p$ of $\Cal^*$, to complete the proof, just note that, for any $x\in M$, 
the given isomorphism from $E^*_x$ onto the space of sections of $\Cal^*_x$ 
is expressed by $\a\mapsto\a|_{\Cal_x}$\,, for any $\a\in E^*_x$\,.
\end{proof} 

\indent 
Let $(Y,\p,\Cal)$ be a regular almost twistorial structure on $M$. To try to describe the geometry induced on $M$, 
let $\nabla$ be a $\r$-connection (locally defined) on $E$\,. By Proposition \ref{prop:almost_Ward}\,, 
$\nabla$ corresponds to a partial connection on $\p^*E$\,, over $\Cal$\,, which we shall denote by $\p^*\nabla$.  
On denoting by $p:\p^*E\to\mathcal{E}$ the projection, we obtain a section $b$ of $\mathcal{E}\otimes\bigl(\otimes^2\Cal^*\bigr)$ 
given by $b(s,t)= p\bigl((\p^*\nabla)_st\bigr)$ for any local sections $s$ and $t$ of $\Cal$\,. 

\begin{prop} \label{prop:suff_cond} 
Suppose that the following two conditions are satisfied:\\ 
\indent 
\quad{\rm (1)} $H^1\bigl(\Cal_x\otimes\bigl(\otimes^2\Cal^*_x\bigr)\bigr)=0$\,, for any $x\in M$\,;\\ 
\indent 
\quad{\rm (2)} $H^0\bigl(\Cal^*_x\bigr)\otimes H^0\bigl(\Cal^*_x\bigr)\to H^0\bigl(\otimes^2\Cal^*_x\bigl)$ is surjective, for any $x\in M$,  
and of constant rank.\\ 
\indent 
Then, locally, there exists a $\r$-connection $\widetilde{\nabla}$ on $E$ such that $\Cal$ is preserved by $\p^*\widetilde{\nabla}$ 
(that is, $(\p^*\widetilde{\nabla})_st$ is a local section of $\Cal$\,, for any local sections $s$ and $t$ of $\Cal$\,). 
\end{prop} 
\begin{proof} 
Firstly, tensorise \eqref{e:rho_induced} with $\otimes^2\Cal^*$, restrict to each fibre of $\p$\,, 
and then pass to the associated cohomology exact sequences. Then conditions (1)\,, (2)\,, and \cite[Theorem 2.3]{KodSpe-I_II} 
imply that $h^0\bigl(\mathcal{E}_x\otimes\bigl(\otimes^2\Cal^*_x\bigr)\bigr)$ does not depend of $x\in M$. 
Consequently, $b$ is induced (locally) by a section of $E\otimes\bigl(\otimes^2\Cal^*\bigr)$\,, and (applying, again, (2)\,) we deduce that 
$b$ is induced by a section $B$ of $E\otimes\bigl(\otimes^2E^*\bigr)$\,; that is, $b=p\circ\bigl((\p^*B)|_{\Cal\otimes\Cal}\bigr)$\,.\\ 
\indent  
Then the $\r$-connection $\widetilde{\nabla}$ on $E$ given by $\widetilde{\nabla}_st=\nabla_st-B(s,t)$\,, for any local sections 
$s$ and $t$ of $E$\,, is as required. 
\end{proof} 

\begin{rem}  
In Proposition \ref{prop:suff_cond}\,, if the distribution $\Cal$ is one-dimensional then (1) implies that $(Y,\p,\Cal)$ is regular. 
Also, if $\Cal_x^*$ is very ample, for any $x\in M$, then (2) is automatically satisfied, provided that (1) holds. 
\end{rem} 

\begin{cor} \label{cor:suff_cond} 
Let $\p:Y\to M$ be a proper surjective submersion, let $\Cal$ be a distribution on $Y$, and let $L$ be a line bundle over $Y$ such that:\\ 
\indent 
\quad{\rm (a)} $\Cal\cap({\rm ker}\dif\!\p)=0$\,;\\ 
\indent 
\quad{\rm (b)} $\Cal\otimes L$ restricted to each fibre of $\p$ is trivial;\\ 
\indent 
\quad{\rm (c)} $L|_{Y_x}$ is very ample, with zero first cohomology group, where $Y_x=\p^{-1}(x)$\,, for any $x\in M$.\\ 
\indent 
Then $(Y,\p,\Cal)$ is a regular almost twistorial structure and, locally, there exist  
$\r$-connections $\nabla$ on $E$ such that $\Cal$ is preserved by $\p^*\nabla$.  
Furthermore, if some fibre of $Y$ is of dimension at least two and has zero first Betti number then $(Y,\p,\Cal)$ is maximal. 
\end{cor}  
\begin{proof} 
This follows quickly from Example \ref{exm:lin_Y-str} and Proposition \ref{prop:suff_cond}\,. 
\end{proof} 

\begin{exm} 
Let $(Y,\p,\Cal)$ be an almost twistorial structure on $M$ such that $\Cal$ is one-dimensional. Assume that, for any $x\in M$, the map 
$\phi_x:Y_x\to PT_xM$\,, $\phi_x(y)=\dif\!\p\bigl(\Cal_y\bigr)$\,, $(y\in Y_x)$\,, is a normal embedding, and $H^1(\Cal^*_x)=0$\,, where 
$\Cal_x=\Cal|_{Y_x}$\,.\\ 
\indent 
On identifying $Y$ with its image, through $\phi=(\phi_x)_{x\in M}$\,, Corollary \ref{cor:suff_cond} implies that, locally, 
there exist partial connections $\nabla$ on some distribution on $M$ such that the projection through $\p$ of any leaf of $\Cal$ 
is a geodesic of $\nabla$. Consequently, the following hold, as well:\\ 
\indent 
\quad1) If the tangent space at some point of a geodesic of $\nabla$ belongs to $Y$ then this holds for all of the points of the geodesic.\\ 
\indent 
\quad2) $\Cal$ is given by the tautological lifts of the geodesics of $\nabla$ whose tangent spaces belong to $Y$.\\ 
\indent 
For concrete examples, take $Y$ to be $PTM$ or the bundle of isotropic directions of a conformal structure on $M$ 
(see \cite{LeB-thesis}\,, \cite{LeB-nullgeod} for different proofs).   
\end{exm}

\section{Regular twistorial structures} \label{section:regular_twist}

\indent  
Recall \cite{PanWoo-sd} that, a \emph{twistorial structure} on $M$ is an almost twistorial structure $(Y,\p,\Cal)$ on $M$ with 
$\Cal$ integrable. Suppose, further, that there exists a surjective submersion $\psi:Y\to Z$ satisfying:\\ 
\indent 
\quad(1) ${\rm ker}\dif\!\psi=\Cal$\,,\\ 
\indent 
\quad(2) $\psi|_{Y_x}$ is injective, for any $x\in M$.\\ 
Then $Z$ is the \emph{twistor space} of $(Y,\p,\Cal)$ and the embedded submanifolds $\psi(Y_x)\subseteq Z$\,, $(x\in M)$\,, 
are the \emph{twistor submanifolds} (cf.\ \cite{PanWoo-sd}\,).\\ 
\indent 
If $(Y,\p,\Cal)$ is regular then, as in Section \ref{section:regular_a_twist}\,, we denote by $E$ the dual of the direct image, 
through $\p$, of $\Cal^*$. We, also, have a canonical morphism of vector bundles $\r:E\to TM$. 

\begin{cor} \label{cor:Ward} 
Let $(Y,\p,\Cal)$ be a regular twistorial structure with twistor space given by $\psi:Y\to Z$\,. If the fibres 
$\psi$ are simply-connected then, for any Lie group $G$, there exists a functorial correspondence between the following:\\ 
\indent 
{\rm (i)} Principal bundles $(\mathcal{P},Z,G)$ whose restrictions to the twistor submanifolds are trivial.\\ 
\indent 
{\rm (ii)} Principal bundles $(P,M,G)$ endowed with principal $\r$-connections whose pull backs to $(Y,\Cal,\iota)$ are flat 
(partial connections), where $\iota:\Cal\to TY$ is the inclusion. 
\end{cor} 
\begin{proof} 
The correspondence is given by $\psi^*\mathcal{P}=\p^*P$ and the result is a quick consequence of Proposition \ref{prop:almost_Ward}\,.  
\end{proof} 

\indent 
Note that, Corollary \ref{cor:Ward} is just the Ward transformation for regular twistorial structures (compare \cite{Pan-hmhKPW}\,).   

\begin{thm} \label{thm:suff_cond} 
Let $(Y,\p,\Cal)$ be a regular twistorial structure with twistor space given by $\psi:Y\to Z$\,, and let $L$ be a line bundle over $Z$ such that:\\ 
\indent 
\quad{\rm (a)} $\Cal\otimes\psi^*L$ restricted to each fibre of $\p$ is trivial;\\ 
\indent 
\quad{\rm (b)} $(\psi^*L)|_{Y_x}$ is very ample, where $Y_x=\p^{-1}(x)$\,, for any $x\in M$.\\ 
\indent 
Then, locally, there exist  
$\r$-connections $\nabla$ on $E$ such that:\\ 
\indent 
\quad{\rm (i)} $\Cal$ is preserved by $\p^*\nabla$;\\ 
\indent 
\quad{\rm (ii)} $\nabla$ is compatible with the decomposition $E=H\otimes F$, where $H$ is the dual of the direct image (through $\p$) 
of $\psi^*L$\,, and $F$ is the direct image of $\Cal\otimes\psi^*L$\,.\\ 
\indent  
Furthermore, if some fibre of $Y$ is of dimension at least two and has zero first Betti number then $(Y,\p,\Cal)$ is maximal. 
\end{thm}  
\begin{proof} 
Assertion (ii) means that $\nabla=\nabla^H\otimes\nabla^F$, 
where $\nabla^H$ and $\nabla^F$ are $\r$-connections on $H$ and $F$, respectively. 
It is sufficient to prove that we can choose $\nabla^H$ and $\nabla^F$ such that (i) is satisfied.\\ 
\indent 
Let $\nabla^F$ be any $\r$-connection (locally defined) on $F$ and we define $\nabla^H$ as follows. 
Firstly, note that, $\psi^*L$ is endowed with a flat partial connection, over $\Cal$\,, whose 
covariantly constant local sections are the pull backs by $\psi$ of the local sections of $L$\,. Consequently, $\p^*H$ is canonically endowed 
with a partial connection, over $\Cal$\,. By Proposition \ref{prop:almost_Ward}\,, this corresponds to a $\r$-connection $\nabla^H$ on $H$.\\ 
\indent 
Now, the proof follows from the obvious identification $\Cal=\psi^*(L^*)\otimes\p^*F$ and 
the fact that, locally, we may choose sections $u$ of $\psi^*(L^*)$ such that 
$\bigl(\p^*\nabla^H\bigr)_{u\otimes v}u=0$\,, for any $v\in\p^*F$ (the latter fact is a consequence of the definition of $\nabla^H$). 
\end{proof} 

\begin{rem}  \label{rem:gen_rho_q} 
With the same notations as in the proof of Theorem \ref{thm:suff_cond}\,, 
we have an obvious embedding $Y\subseteq PH$. Let $\mathcal{L}$ be the dual of the tautological line bundle over $PH$. 
On denoting by $\p$, also, the projection from $PH$ onto $M$, 
note that, $\mathcal{L}|_Y=\psi^*L$ and $\mathcal{L}^*\otimes\p^*F$ is a subbundle of $\p^*E$\,.\\ 
\indent 
Now, $\nabla^H$ induces a $\r$-connection on $PH$\,; that is, a morphism of vector bundles $c:\p^*E\to T(PH)$ which when 
composed with the morphism $T(PH)\to\p^*(TM)$\,, induced by $\dif\!\p$, gives $\p^*\r$\,.\\ 
\indent  
We define $\mathcal{D}=c\bigl(\mathcal{L}^*\otimes\p^*F\bigr)$\,, and we claim that $\mathcal{D}|_Y=\Cal$\,.  
Indeed, this follows from the fact that, locally, we may choose local sections $u$ of $\mathcal{L}^*|_Y$ such that 
$\bigl(\p^*\nabla^H\bigr)_{u\otimes v}u=0$\,, for any $v\in(\p^*F)|_Y$\,.\\ 
\indent 
Therefore $\Cal$ can be retrieved from $(E,\r,Y,\nabla)$\,. This applies, for example, when $Z$ is endowed with a 
locally complete family of Riemann spheres, embedded with nonnegative normal bundles (see \cite{Pan-qgfs} for a different proof).  
\end{rem}

\end{document}